\documentclass[11pt]{amsart}

\usepackage{palatino} 
\usepackage{hyperref}
\usepackage{tikz}
\usetikzlibrary{positioning}

\usepackage{amsmath,amssymb,amsthm,verbatim, amsfonts,amscd,flafter,epsf, epsfig,graphicx,verbatim,pinlabel,mathrsfs}
\usepackage[all]{xy}
\usepackage{epsf}
\usepackage[abs]{overpic}
\usepackage{epstopdf}
\usepackage{graphicx}

\usepackage{xcolor}
\usepackage{mathtools}
\usepackage{scalerel}

\definecolor{bred}{rgb}{1,0,.2}
\definecolor{blue}{rgb}{0,0,1}

\newtheorem{theorem}{Theorem}[section]
\newtheorem{lemma}[theorem]{Lemma}

\theoremstyle{definition}
\newtheorem{definition}[theorem]{Definition}

\def\leave#1{{}}

\def\R{\mathbb{R}}

\def\ms#1{{\textcolor{olive}{#1}}}

\title{A Reconstruction Lemma for Enhanced Bypass Sequences}

\author[Matthias Scharitzer]{Matthias Scharitzer}
\address{Uppsala University}
\email{matthias.scharitzer@math.uu.se}

\begin{document}

\maketitle

\begin{abstract} 
In the appendix of \cite{LSV}, the authors produce an invariant for contact structures on $\Sigma \times [0,1]$. In this paper, we prove a result announced there which states that this is in-fact a full invariant. As an application, we provide an independent proof of Eliashbergs theorem and show that various construction methods for contact structures are equivalent.

\end{abstract}

\section{Introduction}

\subsection{Background, notation and conventions}

We do not assume familiarity with the techniques used in \cite{LSV} and reintroduce the necessary language again in this paper. For the sake of brevity, we assume the reader is familiar with the basics of convex surface theory in low-dimensional contact geometry as presented in, e.g. the frankophone papers \cite{Gi,Gi2,Gi3}. The lecture notes \cite{Hcon,Econ} are also great resources. We also make use of the language of dynamical systems on $2d$-surfaces, we suggest \cite{Sot,KS} as useful resources.

In the following, we will consider $3$-dimensional contact manifolds $(M,\xi)$ where $\xi$ is assumed to be positive and co-orientable. If such a choice of co-orientation is important we denote it by $\alpha$. $\Sigma$ will denote a compact oriented surface, possibly with non-empty boundary. Often, we will take $M=\Sigma \times [0,1]_t$ and write $\Sigma_t$ for a specific $\Sigma \times \{t\}$. We reserve the $s$-variable for $1$-parameter families of contact structures $\xi_s$.

We denote characteristic foliations by $\mathcal{F}$ and representatives by $X$. Note that in our context foliations may be singular and are assumed to be oriented. If the boundary of $\Sigma$ is non-empty then a $0$ or $1$-parameter family of characteristic foliations $\mathcal{F}_t$ is assumed to fulfill the following properties: There is a set of annuli $\mathcal{A}_i$ so that $\partial \Sigma \subset \bigcup\partial \mathcal{A}_i$, the characteristic foliations $\mathcal{F}_t$ are Morse-Smale on each $\mathcal{A}_i$, independent of $t$ and parallel to the boundary of each $\mathcal{A}_i$.

In contrast to the usual conventions, we denote a singularity with a positive determinant as a node which contains several different distinct models from the dynamical systems point-of-view but this distinction is not necessary from the contact geometric side. Similarly, we call a singularity with a negative determinant a saddle. Since our foliations are characteristic, the traces of these singularities are assumed to be non-vanishing and so they have a well-defined sign. For the notation of these generic signed singularities in the graphical calculus, see Figure \ref{fig:Singularity}. An orbit connecting two singularities of opposite sign, emanating from a negative singularity, is called retrograde.

\begin{figure}
    \centering
    \includegraphics[width=1.25\linewidth,trim={7cm 0 0 0} ]{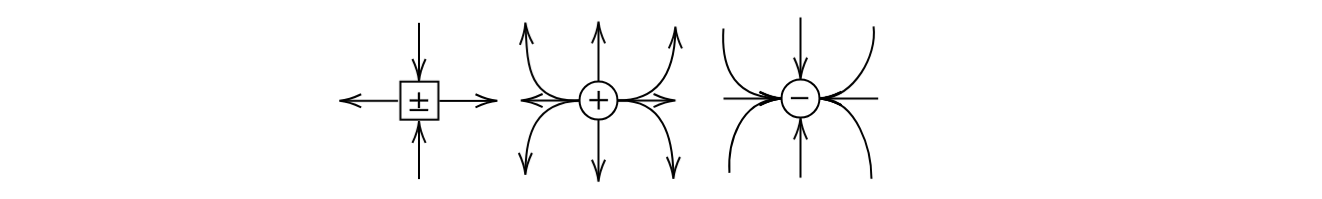}
    \caption{From left to right: A saddle, a positive node and a negative node.}
    \label{fig:Singularity}
\end{figure}

We denote by $\Gamma \subset \Sigma$ multicurves, called dividing curves, that are properly embedded and separating which means that they decompose $\Sigma$ into two (not necessarily connected) subsets. One-parameter families of these are denoted by $\Gamma(r)$ where $r$ usually corresponds to some subset of the $t$-parameter space. We denote by $c$ embedded arcs, called bypass arcs, on $\Sigma$ which are associated to a particular $\Gamma$ which transversely intersect $\Gamma$ in exactly $3$ points and whose endpoints lie on $\Gamma$.

When we write that $\xi_s$ is a contact isotopy on $\Sigma \times [0,1]$ then we mean that the characteristic foliations on $\Sigma_0$ and $\Sigma_1$ remain constant and the characteristic foliations on a neighborhood of the boundary of $\Sigma$ remains invariant throughout. Using Mosers stability trick to obtain flows from contact isotopies shows that these restrictions yield well-defined isotopies of $\Sigma \times [0,1]$ that induce the $\xi_s$.

\subsection{From Characteristic Foliations to Convex Surfaces}

Giroux \cite{Gi,Gi2,Gi3} started a program of studying $3$-dimensional contact manifolds $(M,\xi)$ by probing them with oriented surfaces $(\Sigma,\omega)$. Two of his earliest results are the important Reconstruction Lemmas:
\begin{lemma}(Local Reconstruction Lemma \cite[II. Proposition 1.2.]{Gi})  \label{closed local reconstruction} \label{partial Reconstruction Lemma}

    Let $\mathcal{F}$ be a characteristic foliation. Then there exists a contact structure $\xi$ on $\Sigma \times [0,1]$ so that $\xi$ prints $\mathcal{F}$ onto $\Sigma_{\frac{1}{2}}$. Any two such contact structures are locally contact isotopic.
\end{lemma}

\begin{lemma}(Global Reconstruction Lemma \cite[Lemme 2.1.]{Gi2})

    Let $\mathcal{F}_t$ be a $1$-parameter family of characteristic foliations.
    The set of contact structures $\xi$ which print $\mathcal{F}_t$ onto $\Sigma_t$ for each $t$ is either empty or contractible.
\end{lemma}

Note the slight difference in flavour: Any characteristic foliation locally names a unique isotopy class of contact structures. But this survives globally only if there is at least one contact structure which gives rise to it. It is easy to construct families where these are obstructed.

These results are, however, very unsatisfactory because a characteristic foliation and thus also a movie of characteristic foliations is a very rigid object. However, Giroux separated out a large (open and dense) set of subsurfaces $\Sigma$, called convex surfaces, where contact manifolds foliated by such are particularly flexible:

\begin{theorem} (Realization and Uniqueness Lemma, \cite[Lemme 2.3./Lemme 2.6.]{Gi2}) \label{Convex Flexibility} \label{Interpolation Lemma}

    Let $\Gamma_t$ be a $1$-parameter family of multicurves. Let $\mathcal{F}_t$ be a $1$-parameter family of characteristic foliations which is divided by $\Gamma_t$ for all $t$. Then there is a contact structure which prints $\mathcal{F}_t$ onto $\Sigma_t$ for each $t\in [0,1]$. 

    Furthermore, if one fixes $\mathcal{F}_0$ and $\mathcal{F}_1$ which are divided by $\Gamma_0$ and $\Gamma_1$ then the space of contact structures which print $\mathcal{F}_0$ onto $\Sigma_0$ and $\mathcal{F}_1$ onto $\Sigma_1$ and whose induced characteristic foliations $\mathcal{F}_t$ are divided by $\Gamma_t$ for each $t$ is non-empty and connected.
\end{theorem}

\subsection{Beyond Convexity: Giroux Normal Form and Enhanced Bypass Sequences}

In this section, we wish to study flexibility of contact structures which are not foliated by convex surfaces.

\begin{lemma} (\cite[Proposition 2.4.]{Gi2})

    Let $\Sigma$ be a compact surface with Legendrian boundary. Assume the characteristic foliation fulfills the Poincaré-Bendixson property. Then the following are equivalent:
    \begin{enumerate}
        \item $\Sigma$ is convex.
        \item There are no orbits flowing from a negative singularity to a positive singularity and there are no degenerate closed orbits.
    \end{enumerate}
\end{lemma}

In particular, the open and dense set of Morse-Smale foliations fulfill this property. Using this description, some foliation-theory and rather involved constructions one can produce the following standard form for contact forms:

\begin{definition}
    Let $\xi$ be a contact form on $\Sigma \times [0,1]$ so that $(1)$ the movie of characteristic foliations is generic as a $1$-parameter family of foliations and $(2)$ $\Sigma_t$ is convex except at finitely many times $t_1,\dots,t_n$ where its characteristic foliation is Morse-Smale except for exactly one retrograde connection. We say that $\xi$ is in Giroux normal form.
\end{definition}

\begin{theorem}(Giroux \cite[Lemme 15]{Gi3}) \label{thm:GNF}

    Let $\xi$ be a contact structure so that $\Sigma_0$ and $\Sigma_1$ are convex and the induced characteristic foliations are Morse-Smale. Then $\xi$ is contact isotopic to a contact form in Giroux normal form.
\end{theorem}

For an exposition of this construction, see the appendix of \cite{LSV}. Now to such a Giroux normal form, one can use the following procedure to assign a combinatorial invariant:

Girouxs Crossing Lemma \cite[Lemme 2.13.]{Gi2} tells us that around a retrograde connection between two saddles the characteristic foliations evolve as in Figure \ref{fig:Crossing-Lemma}. The red arc drawn there allows us to define something which we call a bypass arc. The name is, a priori, misleading because the arc we draw is not even Legendrian. We will resolve the origin of this name in Section \ref{Sec:Application 2}, so finally this should cause no confusion.

\begin{figure}
    \centering
    \includegraphics[width=1.25\linewidth,trim={5cm 1cm 0 1cm} ]{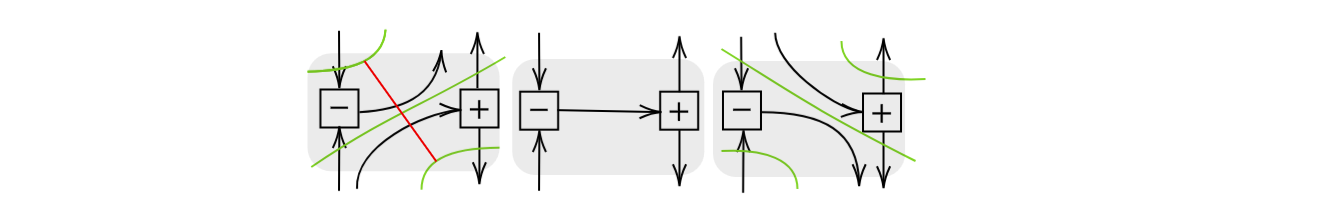}
    \caption{From left to right the evolution of a retrograde connection in a positively co-orientable contact manifold. The green lines indicate the dividing curve and the red arc is the bypass arc.}
    \label{fig:Crossing-Lemma}
\end{figure}

Away from the retrograde connections the set of dividing curves varies smoothly with $t$, together with the assignment of a bypass arc this leads to the following Definition, first introduced in \cite{LSV}:

\begin{definition} An \textit{enhanced bypass sequence}  is a list \[\mathcal{EB} =\big(\Gamma_0(r), c_1, \Gamma_1(r),\dots, c_n, \Gamma_n(r)\big)\] where $\Gamma_i(r)$ is a $1$-parameter family of multicurves such that 
\begin{enumerate}
\item $\Gamma_i(r)$ divides $\Sigma$ into two components;
\item  $c_i$ is a bypass arc on $\big(\Sigma,\Gamma_{i-1}(1)\big)$ such that $c_i$ has a distinguished disc neighbourhood $D_i$;
\item $\Gamma_i(0)=\Gamma_{i-1}(1)$ outside of $D_i$; and
\item within $D_i$, the multicurve $\Gamma_i(0)$ is the result of altering $\Gamma_{i-1}(1)$ by a bypass along $c_i$. 
\end{enumerate}
Below we fix $\Gamma_{start}:=\Gamma_0(0)$ and $\Gamma_{end}:=\Gamma_n(1)$.  
\end{definition}

\begin{definition}\label{def:bseq} 

The contact structure $\xi$ on $\Sigma \times [0,1]$ in Giroux normal form \textit{admits the enhanced bypass sequence} $\mathcal{EB}$ if there exists $\epsilon>0$ such that the following hold:
\begin{enumerate}
\item $\Gamma_{start}$ divides the characteristic foliation on $\Sigma\times \{0\}$ and  $\Gamma_{end}$ divides the characteristic foliation on $\Sigma\times\{1\}$;
\item for $i=1,\dots,n+1$, there is an identification of the $i^{th}$ interval in the list \[[0,t_1-\epsilon]_t,[t_{i-1}+\epsilon, t_i-\epsilon]_t, [t_n+\epsilon, 1]_t\] with $[0,1]_r$ such that $\Gamma_{i-1}(r)$ divides $\Sigma\times \{r\}$; and 
\item\label{it:2} for $t\neq t_i$ in the interval $[t_i-\epsilon,t_i+\epsilon]$, the multicurve $\Gamma_i(0)\vert_{\Sigma\setminus D_i}$ can be extended as a dividing curve for the charactersitic foliation of $\Sigma\times \{t\}$.
\end{enumerate}
\end{definition}

\begin{definition}\label{def:ebseqequiv} Two enhanced bypass sequences $\mathcal{EB}$ and $\mathcal{EB}'$ are \textit{equivalent} if they are related by 
a finite sequence of the following: 
\begin{itemize} 
\item Isotopy

if $\mathcal{EB}$ and $\mathcal{EB}'$ are connected by a 1-parameter family of enhanced bypass sequences with $\Gamma_{start}$ and $\Gamma_{end}$ fixed throughout, then $\mathcal{EB}$ and $\mathcal{EB}'$ are equivalent; 
\item  Far Commutation

  if $D_i\cap D_{i+1}=\emptyset$ and if $\Gamma_i(r)$ and $\Gamma_i'(r)$ are $\ms{r}$-independent,  then 
\[\dots ,\Gamma_{i-1}(r),c_i, \Gamma_i, c_{i+1}, \Gamma_{i+1}(r),\dots \text{ \\ \ and \  \ } \dots, \Gamma_{i-1}(r),c_{i+1}, \Gamma_i', c_{i}, \Gamma_{i+1}(r),\dots\] are equivalent
; and
\item Trivial Insertion

 if $c_T$ is a trivial bypass arc on $(\Sigma,\Gamma_i)$  and $\Gamma_T(r)$ is an isotopy supported near $c_T$ and its trivialising bigon such that $\Gamma_T(1)=\Gamma_i(1)$,  then   \[\dots ,\Gamma_{i}(r),c_i,\Gamma_{i+1}(r),\dots \text{ \ \ \ and \ \ \ }  \dots ,\Gamma_{i}(r),c_T,\Gamma_{T}(r),c_i,\Gamma_{i+1}(r),\dots\] are equivalent.
\end{itemize}
\end{definition}

The following theorem was established in the appendix of \cite{LSV}:

\begin{theorem} (\cite[Theorem A.24.]{LSV}) \label{thm:mainLSV}

Suppose that $\xi$ and $\xi'$ are contact structures on $\Sigma \times I$  in Giroux normal form that are isotopic relative to $\Sigma \times \{0,1\}$. Then enhanced bypass sequences admitted by $\xi$ and $\xi'$ are equivalent.  
\end{theorem}

In this paper, we will extend this to show that this theorem in-fact is true in the following stronger form:

\begin{theorem}(Reconstruction Lemma for EBS) \label{thm:main}

    Let $\mathcal{F}_0$ and $\mathcal{F}_1$ be characteristic foliations on $\Sigma$ which are Morse-Smale. Let $\Gamma_{start}$ and $\Gamma_{end}$ be two dividing curves for $\mathcal{F}_0$ and, respectively, $\mathcal{F}_1$. Then there is a $1:1$-correspondence between the following two sets:
    \begin{enumerate}
        \item Contact isotopy classes of contact structures in Giroux normal form  which print $\mathcal{F}_0$ and $\mathcal{F}_1$ onto $\Sigma_0$ and $\Sigma_1$;
        \item Equivalence classes of EBS with $\Gamma_0(0)=\Gamma_{start}$ and $\Gamma_n(1)=\Gamma_{end}$
    \end{enumerate}
\end{theorem}

\begin{proof}
    Theorem \ref{thm:mainLSV} establishes that there is a well-defined map from $(1)$ to $(2)$. Lemma \ref{Surjectivity} shows that this map is surjective.
    
    Lemma \ref{Injectivity 1} shows that if two contact structures induce the same EBS then they are contact isotopic. Lemma \ref{Injectivity 2} shows that if $\mathcal{EB}$ and $\mathcal{EB}'$ are related by Isotopy, Far Commutation or Trivial Insertion then there are contact structures $\xi$ and $\xi'$ realising $\mathcal{EB}$ and $\mathcal{EB'}$ which are contact isotopic. These two arguments imply that the map is injective which thus completes the proof.
\end{proof}

\emph{Acknowledgments:} The author would like to thank John Etnyre and Vera Vértesi for helpful conversations on this topic. The author is supported by the Swedish Research Council, VR 2022-06593, Centre of Excellence in Geometry and Physics at Uppsala University and VR 2024-04417, project grant.

\section{Applications} \label{Sec:Applications}

In this section, we will explore two applications of Theorem \ref{thm:main}. Eliashberg's theorem on the uniqueness of tight contact structures on the $3$-ball and a folklore theorem on the equivalence of bypasses to other constructions.

\subsection{Eliashberg's Theorem} 

In this section, we show how to obtain Eliashberg's theorem from Theorem \ref{thm:main}:

\begin{theorem} (\cite[Theorem 2.1.3.]{eliashberg1992contact}) \label{Eliashbergs theorem}
    Let $\mathcal{F}$ be a characteristic foliation on the sphere $S^2$ which is divided by a connected dividing curve. Then there is a unique isotopy class of tight contact structures that print $\mathcal{F}$ onto the boundary.
\end{theorem}

\begin{lemma} \label{Unique tight spheres}
    There is a unique equivalence class of EBS that represent a tight contact structure on $S^2 \times [0,1]$.
\end{lemma}

\begin{proof}
    Giroux's tightness criterion states that a convex sphere $S^2$ has a tight neighborhood if and only if the dividing curve is connected. Thus any admissible EBS for such a contact structure can only contain dividing curves that are connected. This is only possible if all bypass arcs are trivial. Any EBS representing a tight contact structure is thus equivalent to an EBS without any bypass arcs. By Theorem \ref{thm:main}, the only invariant of a tight contact structure is the isotopy class of the movie of dividing curves. 

    To prove that the latter class is unique, we proceed as follows: Let $\Gamma$ and $\Gamma'$ be two $1$-parameter families of oriented embeddings of $S^1$ to $S^2$ with the same endpoints $\Gamma(0)=\Gamma'(0);\Gamma(1)=\Gamma'(1)$. Pick a parametrization of $\Gamma(0)$ and $\Gamma(1)$ and extend these parametrizations arbitrarily to $\Gamma$ and $\Gamma'$. Straighten $\Gamma$ and $\Gamma'$ so that in a neighborhood $[0,\epsilon] \cup [1-\epsilon,1]$ of the boundary they agree as parametrized curves and on $[\epsilon,1-\epsilon]$, they are $1$-parameter families of geodesics. We restrict our attention to $[\epsilon,1-\epsilon]$. We recall that $TS^2 \setminus S^2 \cong \R P^3$ via the round metric. Thus the space of parametrized geodesics is $\R P^3$. However, the parametrization was a choice and not inherent to the movie of dividing curves, thus we can further use the fact that the round metric induces a fibre sequence $S^1 \rightarrow \R P^3 \rightarrow S^2$ and thus we see that the space of unparametrized geodesics is actually $S^2$. Now $\Gamma$ and $\Gamma'$ induce paths on $S^2$ with the same endpoints and thus they must be isotopic to one another. Since the map above produces paths of embedded, oriented curves we observe that the space of embedded co-oriented curves on $S^2$ has trivial $\pi_1$ from which we conclude that there is a unique EBS on $S^2$ with fixed endpoints that can be written without any bypass arcs.
\end{proof}

\begin{proof}[Proof of Theorem \ref{Eliashbergs theorem}]
    Normalize $\xi$ and $\xi'$ to the standard contact structure on a neighborhood $B^3_{Darboux}$ around $0$ and cut $B^3 = S^2 \times I \cup B^3_{Darboux}$. Now the boundary foliations of $S^2 \times I$ agree for both $\xi$ and $\xi'$ which implies that we can calculate their EBS. By Lemma \ref{Unique tight spheres} they must be equivalent and thus by Theorem \ref{thm:main} they are isotopic.
\end{proof}

Similarly, one reobtains the following theorem by instead normalizing the contact structures around $2$ points.

\begin{theorem}(\cite[Theorem 2.1.1.]{eliashberg1992contact})
    There is a unique tight contact structure on $S^3$.
\end{theorem}

\subsection{Equivalence of bypasses, retrograde connections and handle attachments} \label{Sec:Application 2}

An often-used folklore theorem which was first proven in the authors' master thesis \cite{Scharitzer}, states that one can equivalently state  Theorem \ref{thm:GNF}, in terms of bypasses. Namely, that one can isotopy $\xi$ on $\Sigma \times [0,1]$ so that there are finitely many $t_1,\dots,t_n$ and a small $\epsilon$ so that $\Sigma_t$ for $|t-t_i| \geq \epsilon$ is convex and there is a bypass relating $\Sigma_{t_i-\epsilon}$ and $\Sigma_{t_i+\epsilon}$. There is also a (not proven in that work) version where these are instead related by attaching (smoothly but not geometrically cancelling) $1$ and $2$-handles where the handles are given specific local models. 

The geometric set-up for each of these $3$ different constructions (retrograde connections, bypass attachment, handle attachment) can be chosen so that they produce contact structures $\xi$ on $\Sigma \times [0,1]$ so that $\Sigma$ splits into a disk $D$ and its complement $\Sigma'$ which fulfill the following axiomatics:

\begin{enumerate}
    \item $\mathcal{F}_t$ restricted to either $\Sigma'$ or $D$ fulfills the boundary conditions layed out in the beginning;
    \item The contact structure $\xi|_{\Sigma' \times [0,1]}$ is $t$-independent.
    \item $\xi|_{D \times [0,1]}$ is tight and prints a characteristic foliation onto $D_0$ and $D_1$ divided by the dividing curves indicated on the left and middle of Figure \ref{fig:Twisting-Disks}.
\end{enumerate}

\begin{figure}
    \centering
    \includegraphics[width=1.25\linewidth,trim={5cm 1cm 0 1cm} ]{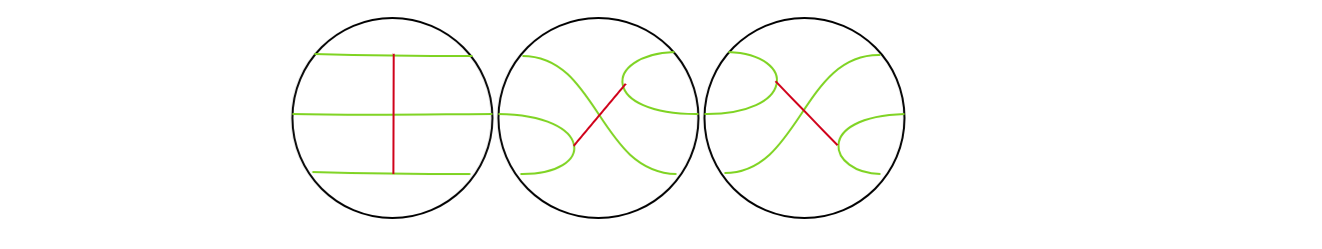}
    \caption{From left to right the three different dividing curves with $3$ components on a disk excluding circular components. Indicated are always the (up to isotopy) unique bypass arcs which are neither trivial nor produce circular components of the dividing curve. Attaching any of them yields the dividing curve to the right.}
    \label{fig:Twisting-Disks}
\end{figure}

The following lemma thus proves that any construction fulfilling the above axiomatics produces the same contact structure and thus one can freely switch between models as desired:

\begin{lemma}
    Let $\Gamma_{start} \subset D^2$ be the multicurve indicated on the left of Figure \ref{fig:Twisting-Disks}. Then there are exactly $2$ equivalence classes of EBS that extend $\Gamma_{start}$ so they represent a tight contact structure: The empty EBS and the EBS which contains exactly one non-trivial bypass arc.
\end{lemma}

\begin{proof}
    Any bypass arc drawable on any of the dividing curves  in Figure \ref{fig:Twisting-Disks} are of any of the following $3$ types: Trivial, it involves all $3$ different components of the dividing curve or it produces a contractible component of the dividing curve. Since we are interested in tight contact structures, we know there are none of the latter. Similarly, we can remove all trivial ones. This leaves some finite sequence of non-trivial bypass arcs where the dividing curves cyclically go through the $3$ dividing curves in Figure \ref{fig:Twisting-Disks}. All of those with the same number of bypass arcs are furthermore isotopic because the isotopy class of a properly embedded curve on a disk with fixed boundary points is contractible. 
    
    Thus, we only need to show that the number of bypass arcs cannot exceed $1$. Assume otherwise, then we can focus on the first $2$. As shown in Figure \ref{fig:OT} these can be commuted past one another which results in an dividing curve with a contractible component. Thus the number of bypass arcs is either $0$ or $1$ and thus we are done.

    \begin{figure}
    \centering
    \includegraphics[width=1.25\linewidth,trim={5cm 1cm 0 1cm} ]{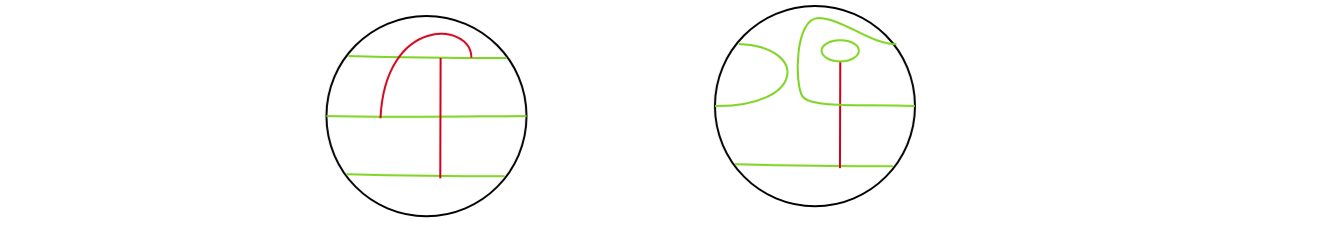}
    \caption{On the left the bypass arc in the middle of Figure \ref{fig:Twisting-Disks} commuted past the bypass arc from the left of Figure \ref{fig:Twisting-Disks}. On the right, the resulting dividing curve of attaching the other bypass arc first. Note the contractible component of the dividing curve.}
    \label{fig:OT}
\end{figure}
\end{proof}

\section{Constructions}

\begin{lemma} \label{enough foliations}

    Let $\Gamma$ be a dividing curve on $\Sigma$ and $c_1,\dots,c_k$ a set of pairwise disjoint bypass arcs on $\Gamma$ where $k$ is any non-negative integer. Then there is a characteristic foliation $\mathcal{F}$ s.t. there are disks $D_1,\dots,D_k$ on $\Sigma$ that fulfill the following:

    \begin{enumerate}
        \item $c_i \subset D_i$;
        \item $\mathcal{F}|_{D_i}$ is as in Figure \ref{fig:Standard};
        \item $\mathcal{F}|_{\Sigma \setminus \bigcup D_i}$ is divided by $\Gamma|_{\Sigma\setminus \bigcup c_i}$.
        \item If a region of the complement of $\Gamma \cup \bigcup c_i$ is a bigon, then $\mathcal{F}$ contains exactly one node.
    \end{enumerate}
\end{lemma}

\begin{figure}
    \centering
    \includegraphics[width=1.25\linewidth,trim={5cm 1cm 0 1cm} ]{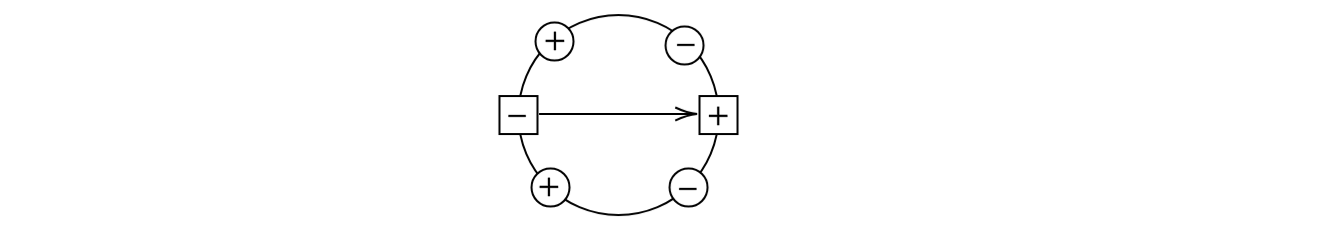}
    \caption{A local model for a standard neighborhood of a retrograde saddle-saddle connection. Note that the position of the saddles on the boundary distinguishes 2 adjacent regions to the bypass arc.}
    \label{fig:Standard}
\end{figure}

\begin{proof}
    We start by picking a disk $D_i$ around each $A_i$ which are pairwise disjoint for $i \neq j$ and endow them with the desired foliation. Now, we only need to extend this prescription onto the complement: Pick one connected component $U$ of the complement of $\Gamma \cup \bigcup D_i$. For sake of definiteness, we assume $U$ is contained in the positive component, the case for the negative component works similarly.
    
    First, we reduce to the case so that $U$ is planar: Pick a boundary component where $U$ meets $\Gamma$ and a handle of $U$. Then as in Figure \ref{fig:Handle} put a closed orbit onto the meridian of the handle of $U$ and two saddles whose separatrices separate the closed orbit from the remaining region of $U$. We apply this to every handle of $U$ and shrink $U$ by removing the regions cut-off by the unstable separatrices of the positive saddles. 
    
    So $U$ is now planar and its boundary components are either pieces of $\Gamma$ or separatrices of positive saddles or two orbits of a positive node (originating from the $D_i$). If there is more than one of the latter kind, then we put a saddle close to this node blocking it. Thus we can assume that $U$ now has either one node or none. If it has none, we place one node into the center. Now, if $U$ is contractible we are done by connecting all remaining orbits which are separatrices of positive saddles or orbits leaving $U$ along $\Gamma$ to the node. This works exactly if the complement of the unstable separatrices of the saddles inside $U$ is contractible. This is the case iff we can cut $U$ open along the unstable separatrices and the boundary becomes connected. If this is not the case then we add positive saddles until it is the case.

    \begin{figure}
    \centering
    \includegraphics[width=1.25\linewidth,trim={8cm 1cm 0 1cm} ]{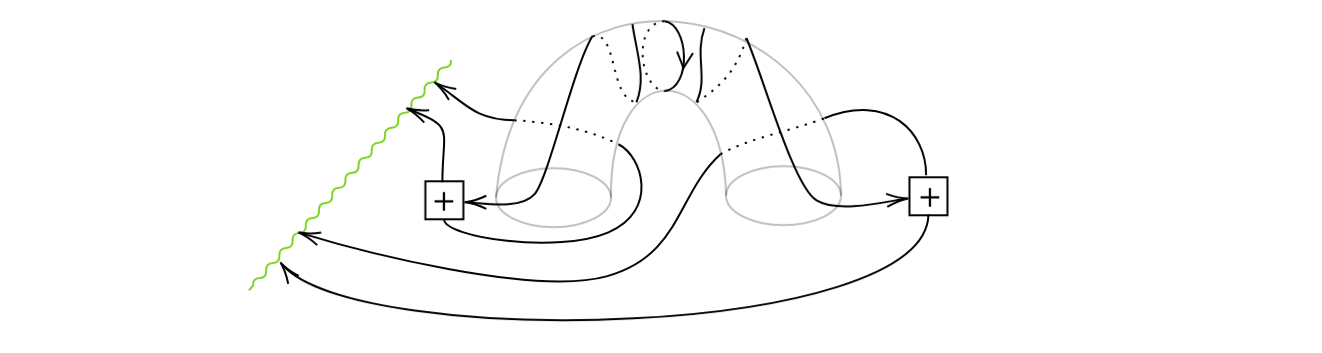}
    \caption{A characteristic foliation locally around a $1$-handle and some piece of any adjacent dividing curve. Note that the closed orbit is repellent.}
    \label{fig:Handle}
\end{figure}

    Following this construction, if $U$ was originally a bigon in $\Sigma \setminus (\Gamma \cup \bigcup c_i) $ then we either added a node into its center or onto the boundary of the disk $D_i$ thus $U$ contains exactly one node.
\end{proof}

\begin{lemma} \label{Analytical flexibility}

    Let $\xi$ and $\xi'$ be two contact structures on $\Sigma \times I$ admitting the same EBS $\mathcal{EB}$ which fulfills the following:

    \begin{enumerate}
        \item $\mathcal{EB}$ contains a single bypass arc.
        \item $\Gamma_i(r)$ is $r$-independent.
    \end{enumerate}

    Then $\xi$ and $\xi'$ are contact isotopic.
\end{lemma}

\begin{proof}
    Assume without loss of generality that $t_1=\frac{1}{2}$, $\epsilon = \frac{1}{4}$ and for $t \in [\frac{1}{4},\frac{3}{4}]$ the characteristic foliation looks as in Figure \ref{fig:Standard} . Denote by $D$ the disk associated to the unique bypass arc $c$. The first step is to modify $\xi$ and $\xi'$ inside $D \times [\frac{1}{4},\frac{3}{4}]$ so that:

    \begin{enumerate}
        \item The surfaces $\Sigma_t$ are convex iff $t \neq \frac{1}{2}$ where the surface is Morse-Smale but for a single retrograde saddle-saddle connection;
        \item $\xi$ and $\xi'$ remain unchanged outside of $D \times [\frac{1}{4},\frac{3}{4}]$;
        \item There exists a smaller disk  $D'$ and $\epsilon' < \frac{1}{4}$ so that $\xi$ and $\xi'$ coincide on  $D' \times [\frac{1}{2}-\epsilon', \frac{1}{2}+\epsilon']$;
        \item $D' \times \frac{1}{2}$ contains the retrograde connection;
    \end{enumerate}

    Once, we have these conditions fulfilled we can apply Lemma \ref{Convex Flexibility} first to $\Sigma \times [\frac{1}{2}-\epsilon',\frac{1}{2}+\epsilon']$ on the complement of $D'$ where the surfaces are convex and divided by the same dividing curves for both $\xi$ and $\xi'$. Thus we can match them up for some $\epsilon '' < \epsilon'$. Finally, we can apply Lemma \ref{Convex Flexibility} again to $\Sigma \times [0,\frac{1}{2}-\epsilon'']$ and $\Sigma \times [\frac{1}{2}+\epsilon'',1]$ where the surfaces are again dividied by the same movies of dividing curves and the characteristic foliations at the boundary coincide. So we can match them up on the full intervall which concludes the construction.

    Thus it remains to construct the first modification. The nuance in this proof is that we do not have access to Giroux's convex flexibility machinery and need to construct a local model by hand where we can control the smooth type of the characteristic foliation inside the disk.

    We will use Giroux's creation lemma \cite[Lemme 16]{Gi2} to introduce cancelling saddle-node pairs for small times: Pick a point $p,p'$ on the retrograde connections of $\xi$ and $\xi'$ and isotopy $\xi'$ inside $D^2 \times [\frac{1}{4},\frac{3}{4}]$ so that $p=p'$ and the movies of characteristic foliations $\xi$ and $\xi'$ agree in a neighborhood of this point. Isotopy $\xi'$ relative to the characteristic foliation so that it agrees with $\xi$ on a smaller neighborhood of $p$. Now on this smaller neighborhood we can pick the same contact forms for both $\xi$ and $\xi'$. Then we apply Giroux's creation lemma multiple times as in Figure \ref{fig:Standard-Disk-Creator}. The key here is to note that, for small times, say, the negative node on the left captures the unstable separatrix of the negative saddle and thus when we create the adjacent positive saddle-node pairs for small enough times then the unstable separatrix will not interact with them. Thus this negative node blocks that saddle from forming retrograde connections as long as it exists. This is enough to construct conditions $2)-4)$, however the foliation $\mathcal{F}_{\frac{1}{2}}$ has too many saddle-saddle connections. We fix this by slightly perturbing the saddle-saddle connections on the outer edge of this disk so that at $t=\frac{1}{2}$ they are no longer connected. This is possible since the contact condition is stable under $C^k$ small perturbations of $\beta$. Thus the constructed movie of foliations also fulfills $1)$.

    \begin{figure}
    \centering
    \includegraphics[width=1\linewidth,trim={1cm 1cm 0 0.5cm} ]{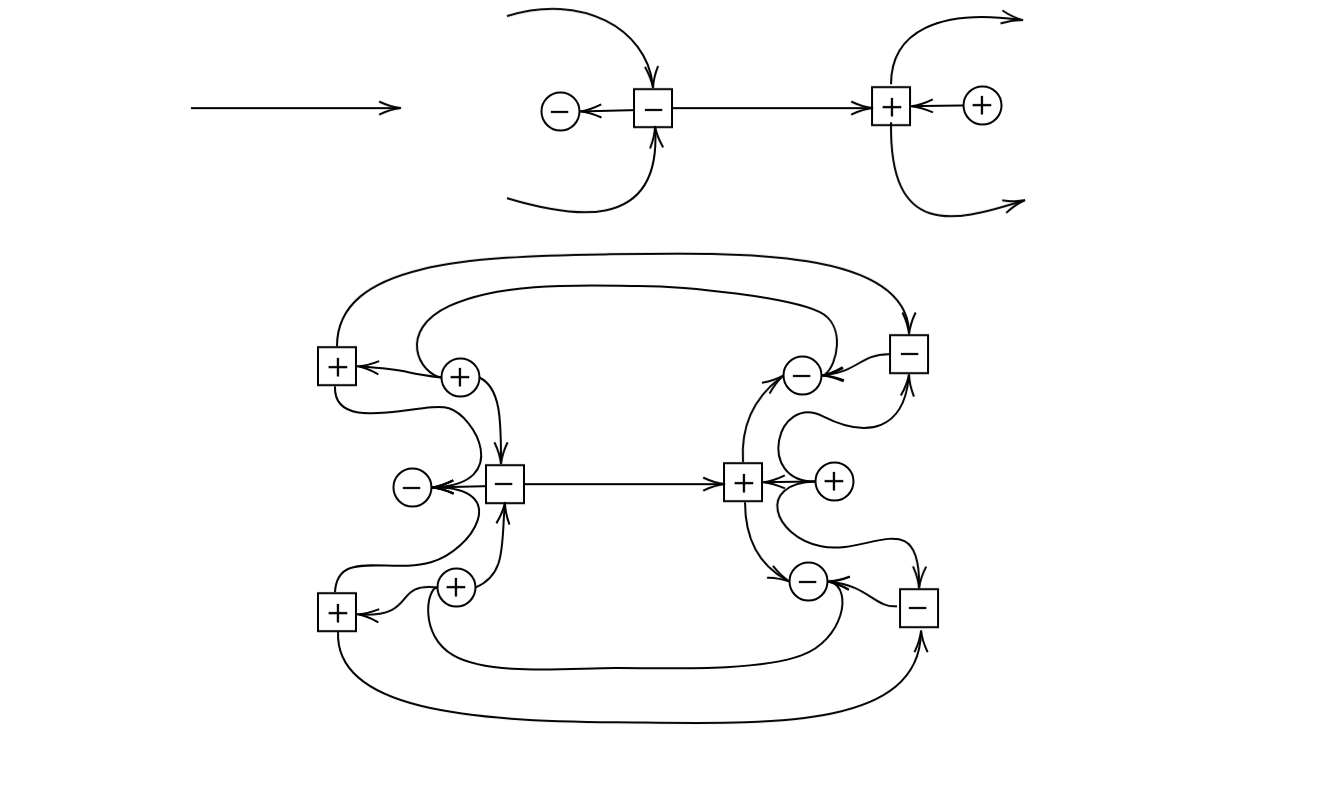}
    \caption{A sequence of saddle-node additions. Top left indicates one piece of the orbit of a retrograde connection. The top right indicates the new characteristic foliation obtained after adding two saddle-node pairs. Finally, (using symmetry of the neighborhood) we obtain the lower characteristic foliation after adding saddle-node pairs along the stable/unstable separatrices of the negative/positive saddle.}
    \label{fig:Standard-Disk-Creator}
\end{figure}
    
\end{proof}

\section{Bijectivity}

\begin{lemma} \label{Surjectivity}
    The map in Theorem \ref{thm:main} is surjective.
\end{lemma}

\begin{proof}
    Let $\mathcal{EB}$ be an arbitrary enhanced bypass sequence with some fixed boundary conditions $\Gamma_{start},\Gamma_{end}$. Choose characteristic foliations for the dividing curves $\Gamma_i(1),\Gamma_{i+1}(0)$ so they are related by $c_i$ as in Figure \ref{fig:Crossing-Lemma}, this is possible by Lemma \ref{enough foliations}. Now we have fixed movies of dividing curves $\Gamma_i(r)$ and boundary foliations. So we may apply Lemma \ref{Interpolation Lemma} to glue these together and obtain that the thus created contact structure admits the given enhanced bypass sequence $\mathcal{EB}$ which proves the Lemma.
\end{proof}

\begin{lemma} \label{Injectivity 1}
    Let $\xi$ and $\xi'$ be two contact structures which admit the same EBS. Then $\xi$ and $\xi'$ are contact isotopic.
\end{lemma}

\begin{proof}
    First, rescale the interval $[0,1]$ via an isotopy $\psi_s:[0,1]\rightarrow[0,1]$ so that the identifications of the domains of the dividing curves $\Gamma_i$ with subsets of $[0,1]$ agree for $\xi$ and $\xi'$. Next, pick characteristic foliations $\mathcal{F}_{t_i\pm \epsilon}$ divided by the appropriate embedded $\Gamma_j$ and apply Lemma \ref{Interpolation Lemma} to $\xi$ and $\xi'$ to all intervals of the shapes $[t_i-\frac{3}{2}\epsilon,t_i-\frac{1}{2}\epsilon],[t_i+\frac{1}{2}\epsilon,t_i+\frac{3}{2}\epsilon]$ so that they print $\mathcal{F}_{t_i-\epsilon}$ onto $\Sigma_{t_i\pm \epsilon}$. Thus, we can apply Lemma \ref{Analytical flexibility} to each $[t_i-\epsilon,t_i+\epsilon]$-slice and Lemma \ref{Convex Flexibility} to the complements which yields the desired result.
\end{proof}

\begin{lemma} \label{Injectivity 2}
     Let $\mathcal{EB}$ and $\mathcal{EB'}$ be two enhanced bypass sequences which are related by iterated Isotopy, Far Commutation or Trivial Insertion. Then there are contact structures $\xi,\xi'$ isotopic to one another and which admit $\mathcal{EB},\mathcal{EB}'$.  
\end{lemma}

\begin{proof}
    By iteration, it is sufficient to show the statement for $\mathcal{EB}$ and $\mathcal{EB'}$ related by either isotopy, Far Commutation or Trivial Insertion. Assume they are related by Isotopy and pick any contact structure $\xi$ which admits $\mathcal{EB}$ then pick an isotopy $\phi_{t,s}$ which pushes each $\Gamma_i(r)$ to $\Gamma'_i(r)$ where the intervals of definition are chosen similarly as in the definition of admissibility of an enhanced bypass sequence to the contact structure $\xi$. By assumption, the endpoints $\Gamma_0(0)$ and $\Gamma_1(1)$ are fixed so we can choose $\phi_{0,s}=\phi_{1,s}=id$ and $\phi_{t,0}=id$. Now for each $s$ $\phi_{t,s}:\Sigma \rightarrow \Sigma$ pushes $\xi$ forward to a $1$-parameter family of contact structures $\xi_{s}$ where $\xi_1$ is as desired.

    Next, we turn to Far Commutation. Without loss of generality, we can assume that $\mathcal{EB}$ has exactly two bypass arcs in its sequence and all the $1$-parameter families of dividing curves can be chosen $r$-invariant. We apply Lemma \ref{enough foliations} to $\Gamma_{start}$ and the two disjoint bypass arcs. This yields a characteristic foliation $\mathcal{F}$ which is divided by $\Gamma_{start}$ in the complement of two standard retrograde disks. Then Lemma \ref{partial Reconstruction Lemma} tells us that this integrates to a contact structure which by Giroux's crossing lemma is convex except at a $t$-value. One can now easily pick two $C^K$-small variations of the vectorfield that pushes one of the retrograde connections slightly up in time or down in time which yields the desired result. These have the desired $\mathcal{EB}$ and $\mathcal{EB}'$.

    Finally, we turn to Trivial Insertion. Without loss of generality, we can assume that $\mathcal{EB}$ has exactly one bypass arc which is the trivial one and that the movie of dividing curves is constant away from the bypass disk. We again appeal to Lemma \ref{enough foliations} which integrates our choice of dividing curve and bypass arc to a characteristic foliation and Lemma \ref{partial Reconstruction Lemma} to turn this into a contact structure with the desired EBS. Now the bypass arc and $\Gamma$ cuts-off a bigoncomponent and the boundary of the standard disk as drawn there is a saddle. So there is a single saddle-node pair close to the disk. Using Giroux's Elimination Lemma \cite[Lemme 2.14.]{Gi2}, we can suppress this saddle-node pair for a small time around the non-convex level set. This modification leaves no further saddles of opposite sign to connect to, so all level sets are now convex as desired.
\end{proof}

\bibliographystyle{alpha}
\bibliography{fob}
\end{document}